\documentclass[11pt,a4paper]{article}
\usepackage{amssymb} %Box
\usepackage{amsmath} %Operator
\usepackage{amsthm} %proof
\usepackage{hyperref}

\newtheorem{theorem}{Theorem}
\numberwithin{theorem}{section}
\newtheorem{proposition}[theorem]{Proposition}
\newtheorem{lemma}[theorem]{Lemma}
\newtheorem{corollary}[theorem]{Corollary}
\newtheorem{conjecture}[theorem]{Conjecture}

\newtheorem{problem}[theorem]{Problem}

\newcommand{\Exp}{\,\mathbb{E}}
\renewcommand{\Pr}{\,\mathbb{P}}
\newcommand{\eps}{\varepsilon}

\DeclareMathOperator{\Bin}{Bin}

\DeclareMathOperator{\ch}{ch}
\DeclareMathOperator{\sep}{sep}

\title{Separation choosability and dense bipartite induced subgraphs\footnote{This research was partly supported by a Van Gogh grant, reference 35513NM.}}

\author{
Louis Esperet
\thanks{Universit\'e Grenoble Alpes, CNRS, G-SCOP, Grenoble, France.
Email: \protect\href{mailto:louis.esperet@grenoble-inp.fr}{\protect\nolinkurl{louis.esperet@grenoble-inp.fr}}.
This author is partially supported by ANR Projects STINT
  (\textsc{anr-13-bs02-0007}) and GATO (\textsc{anr-16-ce40-0009-01}), and LabEx PERSYVAL-Lab
  (\textsc{anr-11-labx-0025}).}
\and
Ross J. Kang
\thanks{Radboud University Nijmegen, Netherlands.
Email: \protect\href{mailto:r.kang@math.ru.nl}{\protect\nolinkurl{r.kang@math.ru.nl}}.
This author is partially supported by a Vidi grant (639.032.614) of the Netherlands Organisation for Scientific Research (NWO).}
\and
St\'ephan Thomass\'e
\thanks{Laboratoire d'Informatique du
  Parall\'elisme, \'Ecole Normale Sup\'erieure de Lyon, France.
Email: \protect\href{mailto:stephan.thomasse@ens-lyon.fr}{\protect\nolinkurl{stephan.thomasse@ens-lyon.fr}}.}
}

\begin{document}

\maketitle

\begin{abstract}
We study a restricted form of list colouring, for which every pair of lists that correspond to adjacent vertices may not share more than one colour. The optimal list size such that a proper list colouring is always possible given this restriction, we call separation choosability.
We show for bipartite graphs that separation choosability increases with (the logarithm of) the minimum degree. This strengthens results of Molloy and Thron and, partially, of Alon.
One attempt to drop the bipartiteness assumption precipitates a natural class of Ramsey-type questions, of independent interest. For example, does every triangle-free graph of minimum degree $d$ contain a bipartite induced subgraph of minimum degree $\Omega(\log d)$ as $d\to\infty$?
\end{abstract}


%Key words: graph colouring, list colouring, triangle-free graphs

\section{Introduction}\label{sec:intro}

The concept of {\em list colouring}, where an adversary prescribes which colours may be used per vertex, was introduced independently by Erd\H{o}s, Rubin and Taylor~\cite{ERT80} and Vizing~\cite{Viz76}.
The starting point for our paper is a restriction to list colouring proposed by Kratochv\'il, Tuza and Voigt~\cite{KTV98a}, which one might imagine makes the colouring task much easier.
Let $G = (V,E)$ be a graph. 
For a positive integer $k$, a mapping $L:V\to \binom{\mathbb Z^+}{k}$ is called a {\em $k$-list-assignment} of $G$, and a colouring $c$ of $V$ is called an {\em $L$-colouring} if $c(v)\in L(v)$ for any $v\in V$.
We say a $k$-list-assignment $L$ has {\em maximum separation} if $|L(u)\cap L(v)| \le 1$ for any edge $uv$ of $G$, and $G$ is {\em separation $k$-choosable} if there is a proper $L$-colouring of $G$ for any $k$-list-assignment $L$ with maximum separation. The {\em separation choosability $\ch_{\sep}(G)$} of $G$ is the least $k$ such that $G$ is separation $k$-choosable.
Kratochv\'il, Tuza and Voigt also considered weaker separation, where $|L(u)\cap L(v)| \le q$ for any edge $uv$, but we only treat the most restrictive non-trivial case --- note the colouring task is trivial if $q=0$.
The original parameter {\em choosability $\ch(G)$} of $G$ demands no separation on the lists, and so we have $\ch(G) \ge \ch_{\sep}(G)$ always.

Let us first review some well-known results for choosability.
It was already known to Erd\H{o}s, Rubin and Taylor~\cite{ERT80} that there are bipartite graphs with arbitrarily large choosability. More specifically, by a connection to the extremal behaviour of the set theoretic ``Property B'', they showed, as $d\to\infty$,
that the complete $d$-regular bipartite graphs $K_{d,d}$ satisfy
\begin{align}\label{eqn:ERT80}
\ch(K_{d,d}) = (1+o(1))\log_2 d.
\end{align}
Somewhat later, Alon demonstrated with a probabilistic argument that just having large minimum degree forces the choosability of a graph to be large. The following considerably generalises the lower bound in~\eqref{eqn:ERT80}.
\begin{theorem}[Alon~\cite{Alo00}]\label{thm:Alon}
There is a constant $C>0$ such that $\ch(G) \ge C\log d$ for any graph $G$ with minimum degree $d$.\qed
\end{theorem}
\noindent
Using containers, Saxton and Thomason~\cite{SaTh15} recently improved this lower bound to $(1+o(1))\log_2 d$ as $d\to\infty$, which is asymptotically optimal by~\eqref{eqn:ERT80}.

One might wonder to what extent the results above remain valid when we impose maximum separation on the lists.
F\"uredi, Kostochka and Kumbhat~\cite{FKK14} already considered the complete $d$-regular bipartite graphs.
By the use of an elegant extremal set theoretic construction (see Lemma~\ref{lem:FKK} below), they strengthened~\eqref{eqn:ERT80} by showing that, as $d\to\infty$, 
\begin{align}\label{eqn:FKK}
\ch_{\sep}(K_{d,d}) = (1+o(1))\log_2 d.
\end{align}
In the first part of this paper, we generalise the lower bound in~\eqref{eqn:FKK}.

\begin{theorem}\label{thm:main}
There is a constant $C>0$ such that $\ch_{\sep}(G) \ge C\log d$ for any bipartite graph $G$ with minimum degree $d$.
\end{theorem}
\noindent
This partially extends Theorem~\ref{thm:Alon}. It also incidentally improves upon a result of Molloy and Thron~\cite{MoTh11} about adaptable choosability, an implication we describe in Subsection~\ref{sub:adapt}.
We give the proof of Theorem~\ref{thm:main} in Subsection~\ref{sub:proof}. The proof combines the construction of~\cite{FKK14} with the probabilistic argument of~\cite{Alo00}, and it requires the assumption that $G$ be bipartite. 

Curiously, the question remains whether a form of Theorem~\ref{thm:main} holds if we drop the bipartiteness assumption. 
\begin{conjecture}\label{conj:chsep}
There is a function $x_1(d)$ that satisfies $x_1(d)\to \infty$ as $d\to\infty$ such that $\ch_{\sep}(G) \ge x_1(d)$ for any graph $G$ with minimum degree $d$.
\end{conjecture}
\noindent
If we only knew Theorem~\ref{thm:Alon} for bipartite $G$, we could still conclude the same way for general $G$ due to the fact that every graph of average degree at least $d$ contains a bipartite subgraph of minimum degree at least $d/2$. This fact and Theorem~\ref{thm:main} are not enough to derive the same conclusion for separation choosability of general $G$: a (bad) list assignment that has maximum separation does not necessarily keep it upon the addition of edges.
On the other hand,
Kratochv\'il, Tuza and Voigt~\cite{KTV98b} have shown that the complete graph $K_{d+1}$ on $d+1$ vertices satisfies $\ch_{\sep}(K_{d+1}) \sim \sqrt{d}$ as $d\to\infty$. (So separation {\em does} help in this situation.)
A natural tack therefore is, in any graph of given minimum degree, to look for either a large complete subgraph or a dense bipartite {\em induced} subgraph.
More precisely, the following if true would imply Conjecture~\ref{conj:chsep}.

\begin{conjecture}\label{conj:cliquebip}
There are functions $x_2(d)$ and $x_3(d)$ satisfying $x_2(d)\to \infty$ and $x_3(d)\to\infty$ as $d\to\infty$ such that
any graph with minimum degree at least $d$ contains a complete subgraph on $x_2(d)$ vertices or a bipartite induced subgraph with minimum degree at least $x_3(d)$.
\end{conjecture}
\noindent
In the second part of this paper, we consider some basic aspects of this conjecture.
Irrespective of separation choosability, this basic Ramsey-type problem is intriguing.

In fact, there is a range of (quantitative) variations on Conjecture~\ref{conj:cliquebip} which are also natural and interesting. Here follow two such variations.
\begin{conjecture}\label{conj:trianglebip}
There is a constant $C>0$ such that any triangle-free graph with minimum degree at least $d$ contains a bipartite induced subgraph of minimum degree at least $C\log d$.
\end{conjecture}
\noindent
If true this would be sharp up to the choice of $C$ (Theorem~\ref{thm:trianglebip}).

\begin{conjecture}\label{conj:girthbip}
There exist $d_0$ and $g_0$ such that any graph of girth at least $g_0$ with minimum degree at least $d_0$ contains a bipartite induced subgraph of minimum degree at least $3$.
\end{conjecture}
\noindent
If $3$ is replaced by $2$ in Conjecture~\ref{conj:girthbip}, then the conclusion is equivalent to containing an even hole. This weaker statement is true with $g_0=4$ and $d_0=3$ by a result of Radovanovi\'{c} and Vu\v{s}kovi\'{c}~\cite[Thm.~1.6]{RaVu13}. Note that even holes can be detected efficiently~\cite{CKS05,CCKV02}, but detection of bipartite induced subgraphs of minimum degree at least 3 is an NP-complete problem~\cite{Hav17}. 

In Section~\ref{sec:cliquebip}, we offer partial progress, towards Conjecture~\ref{conj:trianglebip} especially.
To that end we have found it especially useful that, to find a dense bipartite induced subgraph in a graph of given minimum degree, it suffices to find a good proper colouring. 
\begin{theorem}\label{thm:integral}
Any graph with chromatic number at most $k$ and minimum degree $d$ has a bipartite induced subgraph of minimum degree at least $d/2k$.
\end{theorem}
\noindent
We show this as corollary to a more general result below, Theorem~\ref{thm:fractional}. 
In Subsection~\ref{sub:evidence}, we illustrate how, together with classic results about colourings and stable sets in triangle-free graphs, Theorem~\ref{thm:integral} yields special cases of Conjecture~\ref{conj:trianglebip}: in particular, if the triangle-free graph is nearly regular (Theorem~\ref{thm:Joh96a}) or if it has sufficiently large minimum degree with respect to the number of vertices (Theorem~\ref{thm:AKS80}). By a related but different method, we prove a weaker yet still sharp form of Conjecture~\ref{conj:trianglebip} where edges are permitted for one of the two parts of the sought bipartite subgraph (Theorem~\ref{thm:aksrev}).

\section{Separation choosability of bipartite graphs}\label{sec:sep}

In this section we are primarily concerned with proving Theorem~\ref{thm:main}, but we also discuss some of the implications for another list colouring notion called adaptable choosability, which we discuss in more detail further on.

Although it is not needed for the proof, it may reveal broader context to notice how Theorems~\ref{thm:Alon} and~\ref{thm:main} relate choosability and separation choosability to graph density. Recall the {\em degeneracy} of a graph is the maximum over all subgraphs of the minimum degree; a graph is {\em $d$-degenerate} if its degeneracy is at most $d$. A simple greedy argument yields $\ch(G) \le d+1$ for any $d$-degenerate graph $G$.
Theorem~\ref{thm:Alon} implies that choosability grows with degeneracy: there exists $C>0$ such that
\begin{align}\label{eqn:degenerate}
C\log d \le \ch(G) \le d+1
\end{align}
for any graph $G$ with degeneracy $d$.
Similarly, Theorem~\ref{thm:main} allows us conclude for $\ch_{\sep}(G)$ as in~\eqref{eqn:degenerate} except only for {\em bipartite} $G$ of degeneracy $d$.

We note that by an adaptation of a construction of Kostochka and Zhu~\cite{KoZh08}, for every integer $d$ there is a $d$-degenerate graph that is not separation $d$-choosable.
Moreover, Alon, Kostochka, Reiniger, West and Zhu~\cite{AKRWZ16} have recently shown, with the help of an extremely elegant graph construction, that for every $g$ and $k$ there is a bipartite graph $G$ of girth at least $g$ that is not separation $k$-choosable, every proper subgraph of which has average degree at most $2(k-1)$.

\subsection{Proof of Theorem~\ref{thm:main}}\label{sub:proof}

For the proof of Theorem~\ref{thm:main}, we require the construction~\cite[Corollary~1]{FKK14} used in the proof of~\eqref{eqn:FKK}.
Two hypergraphs $H_1$ and $H_2$ on the same vertex set are \emph{nearly disjoint} if every edge of $H_1$ meets every edge of $H_2$ in at most one vertex.
The stability number $\alpha(H)$ of a hypergraph $H$ is the size of a largest subset of vertices of $H$ that does not contain any edge of $H$.

\begin{lemma}[F\"uredi, Kostochka and Kumbhat~\cite{FKK14}]\label{lem:FKK}
Fix $r\ge2$.
There exist two nearly disjoint $r$-uniform hypergraphs $H_1$ and $H_2$ each on vertex set $[4r^4]$ and having $16r^42^r$ edges such that $\alpha(H_1),\alpha(H_2) < 2r^4$.\qed
\end{lemma}

We essentially substitute this construction into the proof of Theorem~\ref{thm:Alon}.

\begin{proof}[Proof of Theorem~\ref{thm:main}]
Let $G=(V,E)$ be a bipartite graph of minimum degree at least $d$, where without loss of generality $d\ge d_0$ for some sufficiently large fixed $d_0$. Let $V=V_1\cup V_2$ be the bipartition and suppose that $|V_1| \le |V_2|$.
We assume for a contradiction that $\ch(G) \le k$, where $k = \lceil C\log d \rceil-1$ for some fixed $C>0$ chosen strictly smaller than $(1/2-\eps)/\log 2$ with $\eps>0$.
Let $H_1$ and $H_2$ be the hypergraphs certified by Lemma~\ref{lem:FKK} applied with $r=k$.
Write $W=[4k^4]$ for their common vertex set, and $F_1$ and $F_2$ for their respective edge sets.

As already mentioned, the strategy is the same as in~\cite{Alo00}, where we incorporate the use of $F_1$ and $F_2$.
There are two stages of randomness.  In the first stage we choose a small random vertex subset $A$ of $V_1$ and assign lists from $F_1$ to the vertices of $A$ randomly.  After showing that with positive probability there is a subset $A\subseteq V_1$ for which there are many {\em good} vertices in $V_2$ (to be defined below), we fix such a subset $A$ and an assignment of lists.  In the second stage we assign lists uniformly to the good vertices, and only from $F_2$ to ensure separation of the list assignment. Goodness helps to guarantee in the second stage that with positive probability the remaining vertices cannot be coloured from their (random) lists.

Let $p = 1/\sqrt{d}$.  Note that $p < 1/8$ for $d$ sufficiently large.
Let $A$ be a random subset of $V_1$ with each vertex of $V_1$ chosen to be in $A$ independently at random with probability $p$.
Since $\Exp(|A|) = p |V_1|$, by Markov's inequality we have $\Pr(|A| > 2p |V_1|) \le 1/2$.
We define a random list assignment $L_A$ of $G[A]$ as follows: for each $v\in A$ independently, let $L_A(v)$ be a uniformly random element of $F_1$.
We call a vertex $v\in V_2$ {\em good} if for any $f\in F_1$ there is at least one neighbour $v'$ of $v$ such that $v'\in A$ and $L_A(v') = f$.

The probability that a vertex $v\in V_2$ is not good is the probability that for some $f\in F_1$ there is no neighbour $v'$ of $v$ such that $v'\in A$ and $L_A(v')=f$.  For each fixed $f\in F_1$, since there are at least $d$ neighbours of $v$, the probability that there is no neighbour $v'$ of $v$ such that  $v'\in A$ and $L_A(v')=f$ must by the choice of $C$ be at most
\begin{align*}
(1-p/|F_1|)^d \le \exp(-dp/|F_1|) \le \exp(-\sqrt{d}/(16k^42^k)) \le \exp(-d^\eps)
\end{align*}
for $d$ sufficiently large.
Therefore the probability that there is some set $f\in F_1$ that certifies that $v$ is not good must be at most $16k^42^k\exp(-d^\eps)< 1/4$ for $d$ sufficiently large.
It then follows by Markov's inequality that the number of vertices that are not good is at most $|V_2|/2$ with probability strictly greater than $1/2$.  So there exists a set $A\subseteq V_1$ and a list assignment $L_A$ (with lists from the set $F_1$) such that $|A|\le 2p|V_1|$ and the number of good vertices is at least $|V_2|/2$.  Fix such a set $A$ and list assignment $L_A$ for the remainder of the proof.  Let $A^* \subseteq V_2$ be the set of good vertices.

There are at most $k^{|A|}$ possibilities for an arbitrary (proper) $L_A$-colouring of $G[A]$.
We fix one such colouring $c_A$ and show that with high probability there is a $k$-list-assignment of maximum separation extending $L_A$, such that $c_A$ cannot be extended to a proper list colouring of $G[A\cup A^*]$. We define a random list assignment $L_{A^*}$ of $G[A^*]$ by letting, for each $v\in A^*$ independently, $L_{A^*}(v)$ be a uniformly random element of $F_2$.
It is important to notice that, since $G$ is bipartite and $H_1$ and $H_2$ are nearly disjoint, the resulting $k$-list-assignment is guaranteed to have maximum separation.

Consider a vertex $v\in A^*$.
We define $C_v$ to be the set of colours $c$ such that there is some neighbour $v'$ of $v$ such that $v'\in A$ and $c_A(v')=c$. Recall that $W=[4k^4]$ denotes the common vertex set of $H_1$ and $H_2$. We first observe that, since $v$ is good, $C_v$ meets every member of $F_1$.
Thus $W\setminus C_v$ is a stable set of $H_1$, from which we conclude that $|W\setminus C_v| < 2k^4$.
Moreover,  there must be some $f_v \in F_2$ such that $f_v\subseteq C_v$; otherwise, $C_v$ is a stable set of $H_2$ and so $|C_v| < 2k^4$, which in combination with the previous inequality is a contradiction to $|W| = 4k^4$.
By the definition of $C_v$, if the random assignment results in $L_{A^*}(v) = f_v$, then $v$ cannot be properly coloured from its list.  Since the choice of $L_{A^*}(v)$  is independent, the probability that all vertices of $A^*$ can be coloured from their lists is at most
\begin{align*}
(1-1/|F_2|)^{|V_2|/2}\le \exp(-|V_2|/(2|F_2|)).
\end{align*} 
Note that, due to the choices of $k$, $C$ and $p$ and the fact that $|V_2|\ge|V_1|\ge d$,
\begin{align*}
k^{|A|}\exp(-|V_2|/(2|F_2|))
& \le k^{2p|V_1|}\exp(-|V_2|/(2|F_2|)) \\
& \le k^{2p|V_2|}\exp(-|V_2|/d^{1/2-\eps})
< 1
\end{align*}
for $d$ sufficiently large.
It then follows that with positive probability there is a $k$-list-assignment $L$ of $G[A\cup A^*]$ with maximum separation such that there is no proper $L$-colouring of $G[A\cup A^*]$.
\end{proof}

\subsection{Adaptable choosability}\label{sub:adapt}

Given a graph $G=(V,E)$ and a labelling $\ell: E \to [k]$ of the edges, a (not-necessarily-proper) vertex colouring $c: V \to [k]$ is {\em adapted} to $\ell$ if for every edge $e=uv \in E$ not all of $c(u)$, $c(v)$ and $\ell(e)$ are the same value.
We say that $G$ is {\em adaptable $k$-choosable} if for any $k$-list-assignment $L$ and any labelling $\ell$ of the edges of $G$, there is an $L$-colouring of $G$ that is adapted to $\ell$. The {\em adaptable choosability} $\ch_a(G)$ of $G$ is the least $k$ such that $G$ is adaptable $k$-choosable. Every proper colouring is adapted to any labelling $\ell$, so $\ch(G)\ge \ch_a(G)$ always.
This parameter was proposed by Kostochka and Zhu~\cite{KoZh08}. Molloy and Thron~\cite{MoTh11} proved that the adaptable choosability grows with choosability. Recall~\eqref{eqn:degenerate}. More precisely, they proved the following.

\begin{theorem}[Molloy and Thron~\cite{MoTh11}]\label{thm:MoTh}
There is a constant $C>0$ such that $\ch_a(G) \ge C\log^{1/5} d$ for any graph $G$ with minimum degree $d$.\qed
\end{theorem}

We observe separation choosability is at most adaptable choosability.
\begin{proposition}\label{prop:adapt}
For any graph $G$, $\ch_a(G) \ge \ch_{\sep}(G)$.
\end{proposition}
\begin{proof}
Fix $G=(V,E)$ and let $k=\ch_a(G)$. Let $L$ be a $k$-list-assignment of maximum separation. Let $\ell$ be a labelling defined for each $e=uv\in E$ by taking $\ell(e)$ as the unique element of $L(u)\cap L(v)$ if it is nonempty, and arbitrary otherwise. By the choice of $k$, there is guaranteed to be an $L$-colouring $c$ that is adapted to $\ell$. Due to the maximum separation property of $L$ and the definition of $\ell$, the colouring $c$ must be proper.
\end{proof}

By this observation and monotonicity of $\ch_a$ with respect to subgraph inclusion, Theorem~\ref{thm:main} implies the following improvement upon Theorem~\ref{thm:MoTh}.
\begin{corollary}\label{cor:adapt}
There is a constant $C>0$ such that $\ch_a(G) \ge C\log d$ for any graph $G$ with minimum degree $d$.\qed
\end{corollary}
\noindent
This is sharp up to the choice of constant $C$ by~\eqref{eqn:ERT80}.
By~\eqref{eqn:degenerate}, it follows for some $C>0$ that $\ch_a(G) \ge C \log \ch(G)$ for any graph $G$. It has not though been ruled out that $\ch_a(G)$ is $\Omega(\sqrt{\ch(G)})$ in general.

\section{Dense bipartite induced subgraphs}\label{sec:cliquebip}

In this section, we focus mostly on Conjecture~\ref{conj:trianglebip}. That is because we find the triangle-free case the most elegant setting for this class of problems. Furthermore, the techniques carry over routinely to larger forbidden cliques; we briefly discuss this at the end of the section.

As every graph of average degree $d$ contains a (bipartite) subgraph of minimum degree at least $d/2$, and
since we do not worry ourselves about constant factors, we work with average or minimum degree interchangeably.

A first intuition one might have upon meeting this class of problems is that a good proper colouring in a graph should be helpful to find a dense bipartite induced subgraph. If there are few colour classes in such a colouring, then by the pigeonhole principle one expects some pair of the colour classes to have a relatively large number of edges between them. Although this idea in itself does not quite lead directly to the desired induced subgraph (and thus to Theorem~\ref{thm:integral}), this intuition can be formalised and also strengthened as we now show.

Given a graph $G=(V,E)$, let us say that a probability distribution $\mathcal{S}$ over the stable sets of $G$ satisfies property ${\rm Q}^*_r$ if $\Pr(v \in \mathbf{S})\ge r$ for every vertex $v\in V$ and $\mathbf{S}$ a random stable set taken according to $\mathcal{S}$.
Recall that the \emph{fractional chromatic number} of $G$ may be defined as the smallest $k$ such that there is a probability distribution over the stable sets of $G$ that satisfies property ${\rm Q}^*_{1/k}$.

\begin{theorem}\label{thm:fractional}
%Any graph with fractional chromatic number at most $k$ and minimum degree $d$ has a bipartite induced subgraph of average degree at least $d/k$.
Any graph with fractional chromatic number $k$ and average degree $d$ has a bipartite induced subgraph of average degree at least $d/k$.
\end{theorem}
\noindent
By considering the distribution that takes a colour class of a proper colouring uniformly at random, it is easily seen that the fractional is bounded by the usual chromatic number, and so this immediately implies Theorem~\ref{thm:integral}.
As we will see in Subsection~\ref{sub:erdosrenyi}, this lower bound cannot in general be improved by more than a constant factor.

\begin{proof}[Proof of Theorem~\ref{thm:fractional}]
Let $G=(V,E)$ be a graph and let $\mathcal{S}$ be a probability distribution over the stable sets of $G$ with property ${\rm Q}^*_{1/k}$.
Without loss of generality, we may assume that $\Pr(v \in \mathbf{S})= 1/k$ for every vertex $v\in V$ and $\mathbf{S}$ a random stable set taken according to $\mathcal{S}$.
Let $\mathbf{S_1}$ and $\mathbf{S_2}$ be two stable sets taken independently at random according to $\mathcal{S}$. 
Note that $\Exp(|\mathbf{S_1}|) = \Exp(|\mathbf{S_2}|) = n/k$ by the assumption on $\mathcal{S}$.
Moreover, for any edge $e=uv\in E$, the probability that $e$ is in the induced subgraph $G[\mathbf{S_1}\cup \mathbf{S_2}]$ is
\begin{align*}
\Pr(u\in \mathbf{S_1}) \Pr(v\in \mathbf{S_2}) + \Pr(u\in \mathbf{S_2}) \Pr(v\in \mathbf{S_1}) = \frac{2}{k^2}.
\end{align*}
 By linearity of expectation, we have that
\begin{align*}
\Exp\left(|E(G[\mathbf{S_1}\cup \mathbf{S_2}])|-(|\mathbf{S_1}|+|\mathbf{S_2}|)\frac{d}{2k}\right) = \frac{2|E|}{k^2} - \frac{nd}{k^2} = 0.
\end{align*}
The probabilistic method guarantees stable sets $S_1$ and $S_2$ of $G$ such that
\begin{align*}
|E(G[S_1\cup S_2])|\ge (|S_1|+|S_2|)\frac{d}{2k}.
\end{align*}
Discarding the vertices of $S_1\cap S_2$ (if any) yields a bipartite induced subgraph of average degree at least $d/k$.
\end{proof}

\subsection{Subcases of Conjecture~\ref{conj:trianglebip}}\label{sub:evidence}

In this subsection, we discuss some special cases of Conjecture~\ref{conj:trianglebip}. These observations are mainly consequences of Theorem~\ref{thm:fractional} above.

First let us note that, by an induction on the number of vertices, we may assume in Conjecture~\ref{conj:trianglebip} that any proper subgraph of the triangle-free graph has minimum degree less than $d$. (Note that the base case is implied by Tur\'an's theorem.) So the graph may be assumed to be $d$-degenerate.
By Theorem~\ref{thm:fractional}, Conjecture~\ref{conj:trianglebip} is thus implied by a recent conjecture of Harris.

\begin{conjecture}[Harris~\cite{Har16+}]\label{conj:harris}
There is some $C>0$ such that the fractional chromatic number of any $d$-degenerate triangle-free graph is at most $Cd/\log d$.
\end{conjecture}

Next we observe that, due to classic results on the chromatic number of triangle-free graphs and on off-diagonal Ramsey numbers, we may also assume in Conjecture~\ref{conj:trianglebip} that the maximum degree of the graph is large with respect to $d$ and that the minimum degree $d$ is not too large with respect to the number of vertices.

\begin{theorem}\label{thm:Joh96a}
There is a constant $C>0$ such that any triangle-free graph with average degree $d$ and maximum degree $\Delta$ contains a bipartite induced subgraph of
average degree at least $
\frac{d}{C\Delta}\log \Delta$.
\end{theorem}

\begin{proof}
Let $G$ be a triangle-free graph with average degree $d$ and maximum degree $\Delta$.
A result of Johansson~\cite{Joh96a}, recently improved significantly by Molloy~\cite{Mol17}, says that there is a constant $C>0$ such that any triangle-free graph with maximum degree $\Delta$, and thus $G$, has chromatic number at most $C\Delta/\log \Delta$.
The result now follows from Theorem~\ref{thm:fractional}.
\end{proof}

\begin{theorem}\label{thm:AKS80}
There is a constant $C_1$ such that any triangle-free graph on $n$ vertices with average degree $d\ge Cn^{2/3}\sqrt{\log n}$ contains a bipartite induced subgraph of average degree at least $\frac{CC_1}2\log n$ for all $n$ large enough.
\end{theorem}

\begin{proof}
Let $G$ be a triangle-free graph on $n$ vertices with average degree $d\ge Cn^{2/3}\sqrt{\log n}$.
By a classic result of Ajtai, Koml\'os and Szemer\'edi~\cite{AKS80}, the off-diagonal Ramsey numbers $R(3,\ell)$ satisfy $R(3,\ell) = O(\ell^2/\log \ell)$ as $\ell\to \infty$.
Thus, for some $C_1>0$, every induced subgraph of $G$ with at least $n^{2/3}$ vertices contains a stable set of size at least $C_1n^{1/3}\sqrt{\log n}$. We now remove such stable sets (of size at least $C_1n^{1/3}\sqrt{\log n}$) until fewer than $n^{2/3}$ vertices remain in $G$. Note that the subgraph $H$ of $G$ induced by the union of these stable sets has chromatic number at most 
\[\frac{n}{C_1n^{1/3}\sqrt{\log n}} \le \frac{n^{2/3}}{C_1\sqrt{\log n}}.\]
For large enough $n$, the number of edges in $H$ is at least
\[\frac{nd}2-n \cdot n^{2/3} \ge \frac{C}{4} n^{5/3}\sqrt{\log n},\]
and so $H$ has average degree at least $\frac{C}{2} n^{2/3}\sqrt{\log n}$. 
By Theorem~\ref{thm:fractional}, $G$ has a bipartite induced subgraph of average degree at least
$\frac{CC_1}2\log n$.
\end{proof}

In summary, Theorem~\ref{thm:fractional} and the discussion above imply that in Conjecture~\ref{conj:trianglebip} it suffices to consider, for any $\eps>0$, a triangle-free graph $G$ on $n$ vertices of average degree $d$, where $d < \eps n^{2/3}\sqrt{\log n}$, such that $G$ is $d$-degenerate and contains a vertex of degree more than $d/\eps$. Moreover, it suffices to show in these circumstances that the fractional chromatic number of $G$ is at most $Cd/\log d$ for some $C>0$.

\subsection{Dense semi-bipartite subgraphs}\label{sub:semi}

Instead of a dense bipartite \emph{induced subgraph}, we might be happy with a dense bipartite \emph{subgraph} where we only demand (at least) one of the parts induces a stable set.
Given a graph $G=(V,E)$, let us call an induced subgraph $G'=(V',E')$ of $G$ {\em semi-bipartite}
 if it admits a partition $V'=V_1\cup V_2$ such that $V_1$ is a stable set of $G$, and define the average degree of $G'$ with respect to the semi-bipartition as the average degree of the bipartite subgraph $G[V_1,V_2]$ between $V_1$ and $V_2$ (and so we ignore any edges in $V_2$).

In this subsection, we prove the following.

\begin{theorem}\label{thm:aksrev}
Any triangle-free graph of minimum degree $d\ge 1$
contains a semi-bipartite induced subgraph of average degree at least
$\tfrac12\log d$.
\end{theorem}
\noindent
As we will show in Subsection~\ref{sub:erdosrenyi}, this is sharp up to a constant factor.

Near the end of the book of Alon and Spencer~\cite[p.~321--2]{AlSp08}, there is a proof of a result of Ajtai, Koml\'os and Szemer\'edi~\cite{AKS81} that any $n$-vertex triangle-free graph of maximum degree $\Delta$ has a stable set of size at least $n\log \Delta/(8\Delta)$. That proof heavily inspired the proof of the following result. This result directly implies Theorem~\ref{thm:aksrev}.

\begin{lemma}\label{lem:aksrev}
Let $G$ be a triangle-free graph of minimum degree $d\ge 1$, and $\mathbf{S}$ a random stable set of $G$ chosen uniformly. Then $\Exp(\sum_{v\in \mathbf{S}}d(v))\ge \tfrac14 \sum_{v\in G} \log d(v)$.
\end{lemma}

\begin{proof}
For each vertex $v$, let $X_v=d(v)\cdot |\{v\}\cap \mathbf{S}|+|N(v)\cap \mathbf{S}|$.
We claim that
$\Exp(X_v)\ge \tfrac12\log d(v)$.

To see that this claim holds, consider any stable set $T$ of $G-N[v]$ (where $N[v]$ denotes the closed neighborhood of $v$ in $G$). It is enough to prove that for each choice of $T$ as above, $\Exp(X_v\,|\, \mathbf{S}\cap (V(G)-N[v])=T)\ge \tfrac12\log d(v)$. Given a stable set $T$ as above, let $k$ be the number of neighbors of $v$ with no neighbor in $T$. Conditioning on $\mathbf{S}\cap (V(G)-N[v])=T$, there are precisely $2^{k}+1$ equally likely possibilities for $\mathbf{S}$ ($T$ together with $v$, or $T$ together with one of the $2^{k}$ subsets of the set of neighbours of $v$ that have no neighbour in $T$).
It follows that $\Exp(X_v\,|\, \mathbf{S}\cap (V(G)-N[v])=T)=\tfrac1{2^k+1}(d(v)+k2^{k-1})$. 
It can be verified that this last quantity is at least $\tfrac12\log d(v)$ for any $d(v)\ge 1$ and $k\ge 0$; more details are given in the appendix. This proves the claim.

Note that $\sum_{v\in G} X_v = 2\sum_{v \in S} d(v)$. So by linearity of expectation and the claim above $\Exp(\sum_{v \in S} d(v)) = \tfrac12 \sum_{v\in G} \Exp(X_v) \ge \tfrac14 \sum_{v\in G} \log d(v)$.
\end{proof}

\begin{proof}[Proof of Theorem~\ref{thm:aksrev}]
Let $G$ be an $n$-vertex triangle-free graph of minimum degree $d\ge 1$.
Let $\mathbf{S}$ be a stable set of $G$ chosen uniformly at random. By Lemma~\ref{lem:aksrev}, the expected number of edges between $\mathbf{S}$ and its complement is at least $\frac14 n\log d$, and thus the expected average degree of the corresponding semi-bipartite subgraph is at least $\tfrac12\log d$. This proves the existence of the desired semi-bipartite induced subgraph.
\end{proof}

Conjecture~\ref{conj:trianglebip} holds if we can find a distribution $\mathcal{S}$ over the stable sets of the triangle-free graph $G$ that satisfies property ${\rm Q}^*_{\log d/Cd}$. From the proof of Lemma~\ref{lem:aksrev}, it is possible to find a probability distribution $\mathcal{S}$ over the stable sets of $G$ such that $\Exp(|N(v)\cap \mathbf{S}|)\ge \tfrac18 \log d$ or $\Pr(v \in \mathbf{S})\ge \log d/4d$ for every vertex $v$ and $\mathbf{S}$ a random stable set taken according to $\mathcal{S}$. Unfortunately this property does not seem to be powerful enough to prove Conjecture~\ref{conj:trianglebip}.

\subsection{Upper bounds}\label{sub:erdosrenyi}

In this subsection, we prove that Conjecture~\ref{conj:trianglebip} if true would be sharp up to a constant factor.

\begin{theorem}\label{thm:trianglebip}
There are constants $C,C',C''>0$ such that for every large enough $n$ there is a triangle-free graph on $n$ vertices with all degrees between $C'n^{1/3}$ and $C''n^{1/3}$ that contains no semi-bipartite induced subgraph of average degree at least $C\log n$.
\end{theorem}

\noindent
This is a relatively routine probabilistic construction using the binomial random graph, but we include the details for completeness. Theorem~\ref{thm:trianglebip} also certifies sharpness of the bounds in Theorems~\ref{thm:integral},~\ref{thm:fractional},~\ref{thm:Joh96a} and~\ref{thm:aksrev}.

Before continuing, let us make the convenient observation that, if we do not mind constant factors, it suffices to consider only semi-bipartite subgraphs with both parts of equal size.

\begin{proposition}\label{prop:reduction}
Suppose $A,B\subseteq G$ are disjoint with $|A|\ge |B|$ and satisfy that the average degree of $G[A,B]$ is $d$.
Then there exists $A'\subseteq A$ with $|A'|=|B|$ such that the average degree of $G[A',B]$ is at least $d/2$.
\end{proposition}
\begin{proof}
Take $A'$ to be the vertices in $A$ with the $|B|$ largest degrees. 
The number of edges in $G[A',B]$ is at least $\tfrac{|B|}{|A|}|E(G[A,B])|\ge \tfrac{|B|}{|A|} \cdot\tfrac{d}2 (|A|+|B|)$. 
Thus the average degree of $G[A',B]$ is at least $d (|A|+|B|)/2|A| \ge d/2$.
\end{proof}

In the following result we determine up to a constant factor the largest average degree over all semi-bipartite induced subgraphs in the binomial random graph. For Theorem~\ref{thm:trianglebip}, we only need the upper bound.

\begin{proposition}\label{prop:erdosrenyi}
In the binomial random graph on $[n]$ with edge probability $p$, where $np \to\infty$ as $n\to\infty$ and $p<0.99$, the largest average degree of a semi-bipartite induced subgraph is $\Theta(\log np)$ with high probability.
\end{proposition}
\begin{proof}
Fix $\eps>0$.
With high probability,  
the stability number is at most $(2+\eps)\log_b np$, where $b=1/(1-p)$ (cf.~\cite{GrMc75}). By Proposition~\ref{prop:reduction}, we can just consider disjoint vertex subsets $A, B$ with $|A|=|B|=k$, where $k \le (2+\eps)\log_b np$. For a given $k$, the expected number of such $A,B$ where $G[A,B]$ has average degree at least $C\log np$ is as $n\to\infty$ at most
\begin{align*}
&\binom{n}{2k} \binom{2k}{k} \Pr(\Bin(k^2,p) \ge Ck\log np)\\
& \le (en/k)^{2k} \exp(-Ck\log np) < \exp\left(-\frac{C}{2}k\log np\right),
\end{align*}
where we used a Chernoff bound (cf.~\cite[Cor~2.4]{JLR00}) and took a sufficiently large fixed choice of $C>0$. Crudely summing this estimate in the range $C\log np \le k \le (2+\eps)\log_b np$, Markov's inequality implies the chance there is a semi-bipartite induced subgraph with average degree at least $C\log np$ is, as $n\to\infty$, at most
\begin{align*}
(2+\eps)(\log_b np) \exp\left(-\frac{C^2}{2}\log^2 np\right) \to 0
\end{align*}
for $C$ chosen large enough.

For the lower bound, first note by a Chernoff bound (cf.~\cite[Cor~2.3]{JLR00}) that the number of edges is $(1+o(1))n^2p/2$, and so the minimum degree is at least $(1+o(1))np/2$, with high probability. By classic results on the chromatic number of the random graph~\cite{Bol88,Luc91,MaKu90}, with high probability the chromatic number is $(1+o(1))n/(2\log_b np)$. It therefore follows from Theorem~\ref{thm:fractional} that there is a bipartite induced subgraph with average degree at least $(1+o(1))p\log_b np$ with high probability, as required.
\end{proof}

\begin{proof}[Proof of Theorem~\ref{thm:trianglebip}]
Let $p=Dn^{-2/3}$ for some fixed $0<D<2^{-1/4}$.
Consider the binomial random graph on $[n]$ with edge probability $p$.
The expected degree of a vertex is $p(n-1) \sim Dn^{1/3}$. By a Chernoff (cf.~\cite[Cor~2.3]{JLR00}) and a union bound, the minimum degree is less than $Dn^{1/3}/2$ or the maximum degree is more than $3Dn^{1/3}/2$ 
with probability at most $2n\exp(-Dn^{1/3}/12)\to 0$ as $n\to \infty$.
The expected number of triangles is at most $n^3p^3/6=D^3n/6$.
By Markov's inequality, the probability that there are at least $D^3n/3$ triangles is at most $1/2$.
By Proposition~\ref{prop:erdosrenyi} there exists $C>0$ such that, for all $n$ large enough, with positive probability there is a graph $G$ that has minimum degree at least $Dn^{1/3}/2$, maximum degree at most $3Dn^{1/3}/2$, fewer than $D^3n/3$ triangles, and no semi-bipartite induced subgraph of average degree at least $C\log np$.
Assuming $n$ is large enough, fix such a graph $G$ and remove an arbitrary vertex from each triangle. This leaves a triangle-free graph $G^*$ with at least $(1-D^3/3)n>0$ vertices and more than
$Dn^{4/3}/4-(D^3n/3)(3Dn^{1/3}/2) = D(1-2D^4)n^{4/3}/4>0$ edges, so all degrees between $D(1-2D^4)n^{1/3}/4$ and $3Dn^{1/3}/2$.
Moreover by monotonicity $G^*$ contains no semi-bipartite induced subgraph of average degree at least $C\log n\ge C\log np$.
\end{proof}

\subsection{$K_r$-free graphs}

In this section we explain how to extend the results obtained previously for triangle-free graphs to the case of $K_r$-free graphs, with $r\ge 4$.

\medskip

The following result has the same proof as that of Theorem~\ref{thm:Joh96a}, but uses instead the fact that $K_r$-free graphs of maximum degree $\Delta$ have chromatic number at most $200r \Delta \log \log \Delta/\log \Delta$. This was recently proved by Molloy~\cite{Mol17}, improving an earlier result of  Johansson~\cite{Joh96b}.

\begin{theorem}\label{thm:kfree}
There is some constant $C>0$ such that, for every $r\ge 4$, any $K_r$-free graph with average degree $d$ and maximum degree $\Delta$ contains a bipartite induced subgraph of average degree at least $\frac{d \log\Delta}{Cr \Delta \log\log\Delta}$.\qed
\end{theorem}

\noindent
The next result is shown in the same way as Theorem~\ref{thm:AKS80}, but using a slightly more general bound~\cite{AKS80} for the off-diagonal Ramsey numbers: $R(r,\ell) = O(\ell^{r-1}/(\log \ell)^{r-2})$ as $\ell\to \infty$.

\begin{theorem}\label{thm:dense}
For each $r\ge 3$, any $K_r$-free graph $G$ on $n$ vertices with average degree $d=\Omega(n^{1-1/r}(\log n)^{1/(r-1)})$ contains a bipartite induced subgraph of average degree $\Omega(\log n)$ as $n\to\infty$.\qed
\end{theorem}

Note also that a refined analysis in the proof of Lemma~\ref{lem:aksrev}, using~\cite[Lemma 1]{She95} instead of the simple counting argument, yields that there is an absolute constant $C>0$, such that in a $K_r$-free graph $G$, the expected sum of the degrees of the vertices in a random stable set (from the uniform distribution) is at least the sum of $C \tfrac{\log d(v)} {r \log \log d(v)}$, over all the vertices $v\in G$. As a consequence, we obtain the following $K_r$-free analog of Theorem~\ref{thm:aksrev}.

\begin{theorem}\label{thm:sherev}
There is a constant $C>0$ such that for any $r\ge 4$, any $K_r$-free graph of minimum degree $d\ge 3$
contains a semi-bipartite induced subgraph of average degree at least
$C \tfrac{\log d} {r \log \log d}$.\qed
\end{theorem}

For completeness, we merely state two upper bounds relevant to Conjectures~\ref{conj:cliquebip} and~\ref{conj:girthbip}. The proofs routinely adapt the proof of Theorem~\ref{thm:trianglebip}.

\begin{proposition}\label{prop:girthbip}
Fix an integer $g \ge 4$.
There is a constant $C>0$ such that for every large enough $n$ there is a graph on $n$ vertices of girth at least $g$ with minimum and maximum degrees both $\Theta(n^{1/(g-1)})$ that contains no semi-bipartite induced subgraph of average degree at least $C\log n$.\qed
\end{proposition}

\begin{proposition}\label{prop:cliquebip}
Fix an integer $r \ge 3$.
There is a constant $C>0$ such that for every large enough $n$ there is a $K_r$-free graph on $n$ vertices with minimum and maximum degrees both $\Theta(n^{1-2/r})$ that contains no semi-bipartite induced subgraph of average degree at least $C\log n$.\qed
\end{proposition}

\section{Conclusion}\label{sec:conclusion}

In the course of our research, we found alternative derivations of Theorems~\ref{thm:Joh96a}--\ref{thm:aksrev} and of the lower bound in Proposition~\ref{prop:erdosrenyi}. We have omitted these for brevity, but mention here that, like Theorem~\ref{thm:fractional}, they all rely on some appropriate proper colouring. Since it is easy to construct graphs with very high degree bipartite induced subgraphs and also high (fractional) chromatic number, naturally we wonder if there are other, possibly more ``direct'', methods to produce dense bipartite induced subgraphs. In particular, this might help to improve upon the constant factors in our bounds.

Regardless of the eventual status of Conjecture~\ref{conj:trianglebip}, Theorems~\ref{thm:AKS80} and~\ref{thm:trianglebip} hint at the following more refined problem.
\begin{problem}\label{prob:transition}
Fix $\eta\in(0,1)$. As $n\to\infty$, determine the asymptotic infimum $f_\eta(n)$ of the largest minimum degree of a bipartite induced subgraph over all triangle-free graphs of minimum degree $n^\eta$.
\end{problem}
\noindent
Theorem~\ref{thm:trianglebip} implies that $f_{1/3}(n)$ is at most logarithmic in $n$, while a simple modification of Theorem~\ref{thm:AKS80} implies that $f_{2/3+\eps}(n)$, for fixed $\eps>0$, is at least polynomial in $n$. If a guiding ``paradigm'' that couples the triangle-free process with the Erd\H{o}s--R\'enyi process (cf.~e.g.~\cite{BoKe10}) holds as well for our problem, then Proposition~\ref{prop:erdosrenyi} suggests that perhaps $\eta$ being around $1/2$ is the transition point between logarithmic and polynomial behaviour for $f_\eta(n)$.
Note that the $\eta=1$ case is related (via Theorem~\ref{thm:integral}) to a problem of Erd\H{o}s and Simonovits, cf.~\cite{BrTh11+}.

We find it difficult to imagine that Conjecture~\ref{conj:chsep} is not true, but keep in mind that we have only tried one particular approach.

Our study in fact began by considering the separation choosability analogue of a conjecture of Alon and Krivelevich~\cite{AlKr98}. Although our paper is already brimming with conjecture, we still find this worth highlighting. Both this and the original conjecture (for choosability) remain open, and if true would considerably generalise the {\em upper} bounds in~\eqref{eqn:FKK} and~\eqref{eqn:ERT80}, respectively.
\begin{conjecture}
There is a constant $C>0$ such that $\ch_{\sep}(G) \le C\log \Delta$ for any bipartite graph $G$ with maximum degree $\Delta$.
\end{conjecture}

\subsection*{Acknowledgement}

We are grateful to Fran\c{c}ois Pirot for permission to include his improvement on our original Theorem~\ref{thm:fractional}.
We are also thankful for the helpful comments from the anonymous referee.

\subsection*{Note added}

Since the posting of our manuscript to arXiv, the second half of our paper has precipitated a number of significant further developments.
\begin{itemize}
\item Up to a logarithmic factor, Problem~\ref{prob:transition} has been solved, independently, by Cames van Batenburg, de Joannis de Verclos, Pirot and the second author~\cite{CJKP18+} and by Kwan, Letzter, Sudakov and Tran~\cite{KLST18+}.
\item In~\cite{KLST18+}, they moreover proved Conjecture~\ref{conj:cliquebip} (and thus confirmed Conjecture~\ref{conj:chsep}) and have nearly settled Conjecture~\ref{conj:trianglebip} (and thus confirmed Conjecture~\ref{conj:girthbip}) in that they have established $\Omega(\log d/\log\log d)$ bipartite induced minimum degree in $K_r$-free graphs for every fixed $r\ge 3$.
\item Forthcoming work of Davies, de Joannis de Verclos, Pirot and the second author~\cite{DJKP18+} has improved the asymptotic leading constant to $2$ in Theorem~\ref{thm:aksrev}.
\end{itemize}

\bibliographystyle{abbrv}%abbrv
\bibliography{chsep}

\appendix

\section{Technical details in the proof of Lemma~\ref{lem:aksrev}}

In this section, we provide further details of the proof that for any $d\ge 1$ and $0\le k\le d$, it holds that 
\begin{align}\label{eqn:technical}
\frac{d+k2^{k-1}}{2^k+1}\ge \tfrac12\log d.
\end{align}

The inequality~\eqref{eqn:technical} is trivially satisfied if $d=1$, so we can assume that $d\ge 2$. The inequality certainly holds if $d/(2^k+1)\ge \tfrac12\log d$, so we can also assume that $2^k+1> 2d/\log d$, which is equivalent to $k> \log_2(2d/\log d-1)$.

The inequality~\eqref{eqn:technical} also holds provided we can show $k2^{k-1}/(2^k+1)\ge \tfrac12\log d$. Note that the function $x\mapsto x2^{x-1}/(2^x+1)$ is increasing for $x\ge 0$. It follows from our previous assumption that 
\[\frac{k2^{k-1}}{2^k+1} = \tfrac12k(1-(2^k+1)^{-1})
> \tfrac12(1-\tfrac{\log d}{2d})\log_2(\tfrac{2d}{\log d}-1).\]
Thus for~\eqref{eqn:technical} it suffices to establish for $d\ge 2$ that
\[(1-\tfrac{\log d}{2d})\log_2(\tfrac{2d}{\log d}-1) \ge  \log d.\]
It is routinely checked that the function $x\mapsto (1-\tfrac{\log x}{2x})\log_2(\tfrac{2x}{\log x}-1)- \log x$ for $x>1$ has a unique minimum of about $0.30$ at some $x_0\approx 9.74$.

\smallskip

We remark that as $d\to\infty$, the righthand side of inequality~\eqref{eqn:technical} can be improved to $(1-o(1))\log_2 d$.

%%%%%%%%%%%%%%%%%%%%%%%%%%%%%%%%%%%%%%%%%%%%%%%%%%%%%%%%%%%%%%%%%%%%%%%%
\end{document}